\documentclass{article}

\usepackage{amsmath}        
\usepackage{amsfonts}      
\usepackage{amsthm}        
\usepackage{bbding} 
\usepackage{bm}             
\usepackage{graphicx}      
\usepackage{fancyvrb}

\newtheorem{theorem}{Theorem}[section]
\newtheorem{lemma}[theorem]{Lemma}

\newtheorem{example}[theorem]{Example}
\newtheorem{definition}[theorem]{Definition}

\newcommand{\x}{\mathbf{x}}
\newcommand{\e}{\text{e}}
\newcommand{\R}{\mathbb{R}}
\newcommand{\N}{\mathbb{N}}
\newcommand{\cC}{\mathcal{C}}
\newcommand{\BP}{\mathbb{P}}

\newcommand{\BQ}{\mathbb{Q}}
\newcommand{\bN}{{\mathbf N}}

\newcommand{\D}{\mathrm{d}}

\DeclareMathOperator{\E}{\mathbb{E}\,}

\newcommand{\abs}[1]{\left|{#1}\right|}

\newcommand{\comp}[1]{\mathcal{C}^{({#1})}}
\newcommand{\compp}[1]{\mathcal{C}^{{#1}}}

\begin{document}
\pagestyle{myheadings}

\title{Gaussian approximation for
functionals\\ of Gibbs particle processes}

\author{Daniela Novotn\'{a}, Viktor Bene\v{s}\\\\
Department of Probability and Mathematical Statistics,\\ Faculty of Mathematics and Physics, Charles University,\\ Sokolovsk\'{a} 83, 18675 Praha 8, Czech Republic}


\maketitle

\maketitle
\subsection*{Abstract}
In the paper asymptotic properties of functionals of stationary Gibbs particle processes are derived. Two known techniques from the point process theory in the Euclidean space $\R^d$ are extended to the space of compact sets on $\R^d$ equipped with the Hausdorff metric. First, conditions for the existence of the stationary Gibbs point process with given conditional intensity have been simplified recently. Secondly, the Malliavin-Stein method was applied to the estimation of Wasserstein distance between the Gibbs input and standard Gaussian distribution. We transform these theories to the space of compact sets and use them to derive a Gaussian approximation  for functionals of a planar Gibbs segment process.

\noindent{\bf Keywords:} asymptotics of functionals, innovation, stationary Gibbs particle process, Wasserstein distance

\noindent{\bf MSC:} {60D05, 60G55}

\section{Introduction}
\label{sintro}

Recently several papers paid attention to the limit theory of functionals of Gibbs point processes in the Euclidean space, cf. \cite{BYY16, SY13, Tor17, XY15}. In the present paper we are dealing with the question how to develop these results to Gibbs processes of geometrical objects (particles). There are at least three ways how to do it. One natural approach is to extend asymptotic results to Gibbs marked point processes, see e.g. \cite{Mase00}. In applications marks correspond to the geometrical properties of particles, they can be either scalar or vector or particles themselves. In the literature it is often just mentioned that asymptotic results from point process theory can be easily generalized to the marked point processes. This is typically so for processes with independent marks, which may not be the case of Gibbs processes. Another approach is to parametrize some particle attributes and deal with the point process on the parametric space, see e.g. \cite{VB17}. 

In the present paper we are trying to deal directly with particle processes in the sense of \cite{SW08}, defined on the space of compact sets equipped with the Hausdorff distance. Our aim is first to verify that the existence of a stationary Gibbs particle process is guaranteed under analogous conditions as stated by \cite{Der17} for Gibbs point processes. Secondly we find that the methodology of \cite{Tor17} based on Malliavin-Stein method can be developed to Gibbs particle processes. The general background in that paper is formulated on Polish spaces, which covers the space of compact sets. However, the part devoted to Gibbs process is discussed in the Euclidean space only. 

Finally we present examples of two functionals of segment processes in the plane where the Gaussian approximation can be derived using additionally an integral geometric argument.

\section{Preliminaries}\label{sprelim}

Let $\cC^d$ be the space of compact subsets (particles) of $(\R^d,{\cal B}^d),$ 
equipped with the Hausdorff metric and let $\cC^{(d)}=\cC^{d}\setminus\{\emptyset\}.$ Moreover, let ${\cal B}^d,\,{\cal B}(\cC^d)$ denote the Borel $\sigma$-algebras on $\R^d,\,\cC^d,$ respectively. Note that ${\cal B}(\cC^d)$ is equivalent to the Borel $\sigma$-algebra generated by the Fell topology on the space of closed subsets of $\R^d$ restricted to the space $\cC^d$ (cf. \cite[Theorem 2.4.1]{SW08}). Moreover, it can be shown that $\cC^d$ is Polish (cf. \cite[Theorem A.26]{LastPenrose17}). 
Let $\bN^d$ denote the space of all locally finite subsets $\x$ on $\cC^{(d)},$ i.e. 
cardinality $$card\{L\in \x:L\cap K\ne \emptyset\}<\infty$$ for all $K\in \cC^{(d)}.$
We equip this space with the $\sigma$-algebra $${\cal N}^d=\sigma(\{\x\in\bN^d:\,card\{K\in\x:\,K\in B\}=m\},\, B\in {\cal B}(\cC^d)\;{\rm bounded},\; m\in\N).$$  Let $\bN^d_f$ be a subsystem of $\bN^d$ consisting of finite sets.
 
A point process on  $\cC^{(d)}$ (also called particle process) is a random element $$\xi:(\Omega,\mathcal{A},\BP)\longrightarrow (\bN^d,\,{\cal N}^d),$$ its distribution $P_\xi=\BP\xi^{-1}.$
A particle process $\xi $ is called stationary if $P_{\theta_x\xi}=P_\xi$
for each $x\in\R^d$, where for any $\x \in \bN^d$ we set 
$$\theta_x\x =\{K + x\,:K\in\x\},\;K + x\ =\{y+x:y\in K\}.$$
Let $\BQ$ be a probability measure on $\cC^{(d)}$ such that
\begin{equation}\label{cen}
\BQ(\{K\in\cC^{(d)}:c(K)=0\})=1,
\end{equation} 
where $c(K)$ is the centre of the circumscribed ball $B(K)$ of $K$ and $0$ denotes the
origin in $\R^d$.
Define a measure $\lambda$ on $\cC^{(d)}$ by 
\begin{equation}\label{eq1}
\lambda (B)=\int_{\cC^{(d)}}\int_{\R^d}{\bf 1}_{[K+x\in B]}\,\D x \,\BQ(\D K),\; B\in{\cal B}(\cC^d),
\end{equation}
where the inner integration is with respect to the $d$-dimensional Lebesgue
measure $Leb.$ The measure $\lambda $ is invariant under shifts, i.e. $\lambda(B)=\lambda(B+x),\, x\in \R^d.$ We call $\lambda $ the reference measure and $\BQ$ the reference particle distribution.
In the following we make an assumption that
there is some $R>0$ such that

\begin{equation}\label{eq2}
\BQ(\{K\in\cC^{(d)}:B(K)\subset B(0,R)\})=1,
\end{equation} 
where $B(x,R)$ is the closed Euclidean ball with radius $R$ centered at $x\in\R^d.$ 

\subsection{Finite volume Gibbs particle process}
In Gibbs process theory we deal with an energy function as a measurable function \begin{equation}\label{ener}H:\bN^d_f\longrightarrow\R_+\cup\{+\infty\}\end{equation} which will be assumed to be invariant under shifts (stationary), i.e. $H(\x )=H(\theta_x\x),\,$ $x\in \R^d.$ It satisfies $H(\emptyset )<+\infty $ and it is hereditary, i.e. for $\x\in\bN^d_f,\;K\in\x$$$H(\x)<+\infty\implies H(\x\setminus\{K\})<+\infty.$$ A class of energy functions we will deal with is of the form\begin{equation}\label{pp}H(\x )=\sum_{\{K,L\}\subset\x}^{\neq} g(K\cap L),\;\x\in\bN^d_f,\end{equation} where the sum is over pairs of different sets, $g:\compp{d}\longrightarrow\R_+$ is called the pair potential, it is measurable, invariant under shifts such that $g(\emptyset )=0.$ 

In the following we consider a bounded set $\Lambda\subset\R^d$ with $Leb (\Lambda )>0.$ We denote $$\cC^{(d)}_\Lambda=\{ K\in \cC^{(d)};\,c(K)\in\Lambda\}.$$
Let $\bN^d_\Lambda$ be the system of finite subsets of $\cC^{(d)}_\Lambda $ equipped with the trace $\sigma $-algebra ${\cal N}^d_\Lambda .$ Further let
$$\lambda_\Lambda(B)=\int_{\cC^{(d)}}\int_\Lambda{\bf 1}_{[K+x\in B]}\D x\BQ(\D K),\; B\in{\cal B}(\comp{d}_\Lambda)$$ and $\pi_\Lambda $ be the Poisson process on $\cC^{(d)}_\Lambda$ with intensity measure $\lambda_\Lambda.$
We define a finite volume Gibbs particle process on $\Lambda $ with activity $\tau >0,$ inverse temperature $\beta\geq 0$ and energy function $H$ as a particle process with distribution $P^{\tau,\beta }_\Lambda$ on $\bN^d_\Lambda$ given by the Radon-Nikodym density $p$ with respect to $\pi_\Lambda,$ where \begin{equation}\label{den}p(\x )=\frac{1}{Z_\Lambda^{\tau,\beta }}\tau^{N_\Lambda(\x )}\exp(-\beta H(\x )),\;\x\in \bN^d_\Lambda,\end{equation} $N_\Lambda (\x )$ is the number of particles $K\in\x$ with $c(K)\in\Lambda,$ $$Z_\Lambda^{\tau,\beta }=\int_{\bN^d_\Lambda} \tau^{N_\Lambda(\x )}\exp(-\beta H(\x ))\pi_\Lambda (\D\x)$$ is the normalizing constant. 

For any bounded set $\Delta\subset\Lambda,\;Leb(\Delta )>0,\,\Delta^c$ is its complement in $\Lambda$ and for $\x\in\bN^{(d)}_\Lambda $ let $\x_\Delta=\{K\in\x;\;c(K)\in\Delta\}.$ We define $$H_\Delta(\x)=H(\x)-H(\x_{\Delta^c}).$$ The following are Dobrushin-Lanford-Ruelle (DLR) equations for $P_\Lambda^{\tau,\beta }$-a.a. $\x_{\Delta^c}$ we have \begin{equation}\label{dlr}P_\Lambda^{\tau,\beta }(\D\x_\Delta|\x_{\Delta^c})=\frac{1}{Z_\Delta^{\tau,\beta }(\x_{\Delta^c})}\tau^{N_\Delta(\x )}\exp(-\beta H_\Delta(\x ))\pi_\Delta(\D\x_\Delta ),\end{equation} where $$Z_\Delta^{\tau,\beta}(\x_{\Delta^c})=\int\tau^{N_\Delta(\x)}\e^{-\beta H_\Delta(\x)}\pi_\Delta(\D\x_\Delta).$$
The local energy $h$ of $K$ in $\x\in\bN^d_f$ is defined as $$h(K,\x)=H(\x\cup\{K\})-H(\x).$$ The Georgii-Nguyen-Zessin (GNZ) equations follow for any measurable function $f:\cC^{(d)}\times\bN^d_f\longrightarrow\R_+$\begin{equation}\label{gnz}\int_{\bN^d_f}\sum_{K\in\x}f(K,\x\setminus \{K\})P_\Lambda^{\tau,\beta }(\D\x)=\tau\int_{\bN^d_f}\int_{\cC^{(d)}_\Lambda} f(K,\x)\exp(-\beta h(K,\x))\lambda (\D K)P_\Lambda^{\tau,\beta }(\D\x). \end{equation}The GNZ equations characterize the finite volume Gibbs particle process, i.e. if any probability measure on $\bN^d_\Lambda$ satisfies (\ref{gnz}) for any $f$ as stated, then it is equal to $P_\Lambda^{\tau,\beta }.$ The function $$\lambda^*(K,\x )=\tau\exp(-\beta h(K,\x)),\;K\in\cC_\Lambda^{(d)},\,\x\in\bN_\Lambda^d$$ is called the (Papangelou) conditional intensity.

\subsection{Infinite volume Gibbs particle process}
It is verified that the results obtained for point processes in $\R^d$ in \cite{Der17} hold in the particle process case as well. Consider the sequence of windows $$\Lambda_n=[-n,n]^d\subset\R^d,$$ spaces $\cC^{(d)}_{\Lambda_n},$ intensity measures
$\lambda_n=\int\int_{\Lambda_n}{\bf 1}_{[K+x\in .]}\D x\BQ(\D K)$ (for a fixed probability measure $\BQ$ satisfying (\ref{cen}) and (\ref{eq2})), Poisson particle processes $\pi_{\Lambda_n},$ Gibbs point processes $P_{\Lambda_n}^{\tau,\beta },\, n\in\N.$ A measurable function $f:\bN^d\longrightarrow\R$ is called local if there is a bounded set $\Delta\subset\R^d$ such that for all $\x\in\bN^d$ we have $f(\x)=f(\x_\Delta ).$ The local convergence topology on the space of probability measures $P$ on $\bN^d$ is the smallest topology such that for any local and bounded function $f:\textbf{N}^{d}\longrightarrow\R$ the map $P\mapsto\int f\D P$ is continuous. Define a probability measure $\bar{P}^{\tau,\beta}_{\Lambda_n}$ such that for any $n\geq 1$ and any measurable test function $f_1:\bN^d\longrightarrow \R$ it holds
\begin{equation}\label{stat}\int_{\bN^d} f_1(\x )\bar{P}^{\tau,\beta}_{\Lambda_n}(\D\x)=(2n)^{-d}\int_{\Lambda_n}\int_{\bN^d} f_1(\theta_u\x)P^{\tau,\beta}_{\Lambda_n}(\D\x)\D u.\end{equation} It can be shown that the sequence $(\bar{P}^{\tau,\beta}_{\Lambda_n})_{n\geq 1}$ is tight for the local convergence topology (cf. \cite[Chapter 15]{G11}). We denote $P^{\tau,\beta }$ one of its cluster points. Due to the stationarization (\ref{stat}) $P^{\tau,\beta}$ is the distribution of a stationary particle process, in order to show that it satisfies the DLR and GNZ equations one needs to add an assumption. 

The energy function $H$ has a finite range $r>0$ if for every bounded set $\Delta\subset\R^d$ the energy $H_\Delta$ is a local function on $\Delta\oplus B(0,r),$ where $\oplus$ is the Minkowski sum of sets. The finite range property allows to extend the domain of $H$ and $H_\Delta$ from the space $\bN^d_f$ to $\bN^d$ and consequently to define the desired stationary Gibbs particle process.
\begin{definition}\label{df1} Let $H$ be a stationary and finite range energy function on $\bN^d$. A stationary Gibbs particle process is a particle process with distribution $P$ on $\bN^d$ invariant over shifts, such that for any bounded $\Delta\subset\R^d,\;Leb(\Delta)>0,$ for $P$-a.a. $\x_{\Delta^c}$ it holds \begin{equation}\label{dlr1}P(\D\x_{\Delta}|\x_{\Delta^c})=\frac{1}{Z_\Delta^{\tau,\beta }}\tau^{N_\Delta(\x )}\exp(-\beta H_\Delta(\x ))\pi_\Delta(\D\x_\Delta ),\end{equation}$\tau>0,\,\beta\geq 0,$ the denominator is the normalizing constant.\end{definition} For the stationary and finite range energy function the cluster point $P^{\tau,\beta }$ satisfies DLR equations (\ref{dlr1}). Also it satisfies GNZ equations for any measurable function $f:\cC^{(d)}\times\bN^d\longrightarrow\R_+:$\begin{equation}\label{gnz1}\int_{\bN^d}\sum_{K\in\x}f(K,\x\setminus \{K\})P^{\tau,\beta }(\D\x)=\int_{\bN^d}\int_{\cC^{(d)}} f(K,\x)\lambda^*(K,\x)\lambda (\D K)P^{\tau,\beta }(\D\x). \end{equation}Conversely, any measure $P$ on $\bN^d$ which satisfies (\ref{gnz1}) is a distribution of a stationary Gibbs particle process. Then given an hereditary function $\lambda^*$ on $\cC^{(d)}\times \bN^d$ there exists a stationary Gibbs particle process with $\lambda^*$ being its conditional intensity. The uniqueness issue is not investigated in this paper, see \cite{Der17} for more discussion.

In this work, we deal with the conditional intensity of the form
\begin{equation}\label{eq10}
\lambda^*(K, \x):= \tau \exp \left\lbrace - \beta\sum_{L\in\x} g(K \cap L)\right\rbrace, \quad K \in \mathcal{C}^{(d)},\;\x\in\bN^d,
\end{equation}
where $g$ is the pair potential, $\tau>0,\,\beta\geq 0.$
\begin{example}[Planar segment process]\label{ex2.4}
Denote by $S \subset \mathcal{C}^{(2)}$ the space of all segments in $\R^2,\; S_0$ be the subsystem of segments centered in the origin. Fix a reference probability measure $\BQ$ on $S_0,$ which corresponds to $\BQ_\phi\otimes\BQ_L,$ where $\BQ_\phi,\;\BQ_L$ is the reference distribution of directions, lengths of segments, respectively. Thanks to the assumption (\ref{eq2}) $\BQ_L$ has support $(0,2R].$ Set
\begin{equation}\label{eq11}
g(K) = \textbf{1} \lbrace K \neq \emptyset \rbrace, \quad K \in \mathcal{C}^{2},
\end{equation}
and using the previous construction we define the stationary Gibbs segment process $\xi$ in $\R^2$ as a stationary Gibbs particle process with conditional intensity 
$$
\lambda^*(K, \x) = \tau \exp \left\lbrace - \beta \sum_{L\in\x} \textbf{1}\lbrace K\cap L \neq \emptyset\rbrace  \right\rbrace,\quad K \in S, \x \in \textbf{N}^2. 
$$
In fact, $\lambda^*(K, \x)= \tau \e^{-\beta N_{\x} (K)}$, where $N_{\x} (K)$ denotes the number of intersections of $K$ with the segments in $\x$. It has to be mentioned that the reference distribution $\BQ$ need not coincide with the observed joint length-direction distribution of the process, cf. \cite{BVP}.
\end{example}

\section{Generalization of some asymptotic results for Gibbs particle processes}
Our aim is to extend the result \cite[Theorem 5.3]{Tor17} concerning estimates of the bound of the Wasserstein distance between standard Gaussian random variable and functionals of stationary Gibbs point processes in $\R^d$ given by conditional intensity. To do so we will use the general bound given by \cite[Corrolary 3.5]{Tor17} that is formulated for wider class of point processes having conditional intensity. We consider the space $\cC^{(d)}$ of compact sets, conditional intensity (\ref{eq10}) and a stationary Gibbs particle process $\mu $ from Definition \ref{df1}, satisfying (\ref{gnz1}). Behind the presented model there is a probability measure $\BQ$ on $\cC^{(d)}$ satisfying (\ref{cen}) and (\ref{eq2}), defining the reference measure $\lambda $ in (\ref{eq1}). In the following we always mean that a stationary Gibbs particle process has activity $\tau ,$ inverse temperature $\beta ,$ pair potential $g$ and particle distribution $\BQ .$ Thanks to (\ref{eq2}), (\ref{pp}) and assumptions laid on $g$ the finite range property holds.

\subsection{Bounds on Wasserstein distance for functionals of Gibbs particle processes}

The mean value $\E[\lambda^*(K,\mu )],\; K\in\cC^{(d)},$ is called a correlation function.
Sharp lower and upper bound for the correlation function of a Gibbs point process on $\R^d$ can be found in \cite{Stucki14_2}. For our purpose the following simple bounds for the correlation function of a stationary Gibbs particle process are sufficient. 
\begin{lemma}\label{lemma3.1}
Let $\mu$ be a stationary Gibbs particle process given by the conditional intensity of the form \eqref{eq10} with activity $\tau >0$, inverse temperature $\beta \geq 0$, reference particle distribution $\mathbb{Q}$ satisfying \eqref{eq2}, and with pair potential $g$ which is bounded from above by some positive constant $a$. Then there 
exists $b\in [0,\infty )$ such that it holds
\begin{equation}\label{eq13}
\tau (1-\beta b) \leq \mathbb{E}[\lambda^*(K,\mu)] \leq \tau
\end{equation}
for $\lambda$-a.a. $K \in \comp{d}$.
\end{lemma}
\begin{proof}
The stationary process $\mu $ has some intensity measure $\theta $ and particle distribution $\BQ_1$ (typically not equal to $\BQ$). Using the Campbell theorem and the disintegration (\cite{SW08}) we obtain
\begin{align*}
\mathbb{E}[\lambda^*(K,\mu)] & = \tau\mathbb{E}[\exp\lbrace -\beta \sum_{L \in \mu} g(K \cap L)\rbrace] \geq \tau \left(1- \beta \mathbb{E} \sum_{L \in \mu} g(K \cap L)\right)\\
& = \tau \left(1- \beta \int_{\comp{d}} g(K \cap L) \theta(\D L)\right)\\
& = \tau \left(1- \beta \int_{\comp{d}} \int_{\R^d} g(K \cap (L+x))\,\D x \,\mathbb{Q}_1(\D L)\right)\\
& \geq \tau \left(1- \beta a\int_{\comp{d}} Leb (K\oplus\check{L}) \BQ_1(\D L)\right)\\
& \geq \tau (1-\beta b),
\end{align*} where we denote $\check{L} = \lbrace -l : \ l \in L \rbrace$. The support ${\rm supp}\BQ_1\subset {\rm supp}\BQ,$ cf. \cite{BVP}, therefore using (\ref{eq2}) we can choose $b= a (2R)^d \omega_d$, where $\omega_d$ is the volume of the unit ball in $\R^d$.
The upper bound follows from \eqref{eq10}. 
\end{proof}

\begin{definition}
We define the innovation of a Gibbs particle process $\mu$ as a random variable
$$I_{\mu} (\varphi) = \sum_{K \in \mu} \varphi(K, \mu \setminus \{K\}) - \int_{\comp{d}}\varphi(K,\mu)\lambda^*(K,\mu)\lambda(\D K)$$
for any measurable $\varphi: \comp{d} \times \textbf{N}^d \to \R$, for which $|I_{\textbf{x}}(\varphi)|< \infty$ for $\mu$-a.a. $\textbf{x} \in \textbf{N}^d$.
\end{definition}

We are interested in estimates of the Wasserstein distance $d_W$, cf. \cite{Tor17} between an innovation $I_{\mu}$ and a standard Gaussian random variable $Z$. 

\begin{theorem}\label{th3.4}
Let $\mu$ be a stationary Gibbs particle process given by the conditional intensity of the form \eqref{eq10} with activity $\tau >0$, inverse temperature $\beta \geq 0$, reference particle distribution $\mathbb{Q}$ satisfying \eqref{eq2}, and with pair potential $g$ which is bounded from above by some positive constant $a$. Let $\varphi: \mathcal{C}^{(d)} \to \R$ be a measurable function that does not depend on $\textbf{x} \in \textbf{N}^d$  and
$$\varphi \in L^1(\comp{d}, \lambda) \cap L^2(\comp{d},\lambda).$$
Then
\begin{align*}
d_W(I_{\mu}(\varphi), Z) \leq & \sqrt{\frac{2}{\pi}}\sqrt{1-2 \tau (1-\beta b) ||\varphi||^2 _{L^2(\comp{d}, \lambda)}+ \tau^2 ||\varphi||^4 _{L^2(\comp{d}, \lambda)}} \\
& + \tau ||\varphi||^3 _{L^3(\comp{d}, \lambda)} + \sqrt{\frac{2}{\pi}} \tau^2 ||\varphi||^2 _{L^1(\comp{d}, \lambda)} |1-\e ^{-\beta a}| \\
& +  2 \tau^2 ||\varphi||^2 _{L^{2}(\comp{d}, \lambda)} ||\varphi|| _{L^1(\comp{d}, \lambda)} |1-\e ^{-\beta a}| \\
& + \tau^3||\varphi||^3 _{L^1(\comp{d}, \lambda)} |1-\e ^{-\beta a}|^2.
\end{align*}

\end{theorem}

\begin{proof}
First note that in this setting the finite range property holds. We would like to estimate individually terms of the bound in Corollary $3.5$ in \cite{Tor17} (valid on a Polish space). First of all, we need to verify the assumptions of this result. By using the upper bound from Lemma \ref{lemma3.1} and the integrability assumptions on $\varphi$, we can write
$$\int_{\comp{d}} |\varphi(K)| \mathbb{E}[\lambda^*(K,\mu) ] \lambda(\D K) \leq \tau||\varphi||_{L^1(\comp{d}, \lambda)} < \infty $$
and
$$ \int_{\comp{d}} |\varphi(K)|^2 \mathbb{E}\left[   \lambda^*(K,\mu) \right] \lambda(\D K)\leq \tau ||\varphi||^2 _{L^2(\comp{d}, \lambda)}  < \infty$$
and hence, the assumptions are verified.

For simplicity, denote
$$\alpha_2(K,L,\mu) := \mathbb{E}[ \lambda^*(K,\mu) \lambda^*(L,\mu)],$$
$$\alpha_3(K,L,M,\mu) := \mathbb{E}[ \lambda^*(K,\mu) \lambda^*(L,\mu)\lambda^*(M,\mu)]$$
for $K, L, M \in \comp{d}$. Then again based on Lemma \ref{lemma3.1}, we can estimate $\alpha_2$ and $\alpha_3$ as

\begin{align}
\label{eq16}
\begin{split}
 \alpha_2(K, L, \mu) &  \leq \tau^2 ,
\\
 \alpha_3(K, L, M, \mu)  & \leq \tau^3,
\end{split}
\end{align}
for $\lambda$-a.a. $K, L, M \in \comp{d}.$

If we investigate the term $D_K \lambda^*(L, \textbf{x})$ in Corollary $3.5$ in \cite{Tor17}, we obtain
\begin{align*}
& D_K \lambda^*(L, \textbf{x})=\lambda^*(L,\textbf{x}\cup\{K\})-\lambda^*(L,\x)=\\
& = \tau \exp \left\lbrace - \beta\sum_{M\in\x\cup\{K\}} g(L \cap M) \right\rbrace  - \tau \exp \left\lbrace - \beta\sum_{M\in\x}g(L \cap M)\right\rbrace\\
& = \tau \exp \left\lbrace - \beta\sum_{M\in\x} g(L \cap M) \right\rbrace \left(\e^{-\beta g(L \cap K)} -1 \right) \\
& = \lambda^*(L, \textbf{x})\left(\e^{-\beta g(L \cap K)} -1 \right) . 
\end{align*}

Using this expression, we can compute the individual terms. In the first term, we will use estimates \eqref{eq13} and \eqref{eq16} to obtain the bound
\begin{align*}
& \sqrt{\frac{2}{\pi}} \sqrt{1-2 \int_{\comp{d}} |\varphi(K)|^2 \mathbb{E}[  \lambda^*(K,\mu)] \lambda(\D K) + \int_{(\comp{d})^2} |\varphi(K)\varphi(L)|^2 \alpha_2(K,L,\mu) \lambda(\D K)\lambda(\D L) } \\
& \leq \sqrt{\frac{2}{\pi}} \sqrt{1-2 \tau (1-\beta b) \int_{\comp{d}} |\varphi(K)|^2 \lambda(\D K) + \tau^2 \int_{(\comp{d})^2} |\varphi(K)\varphi(L)|^2 \lambda(\D K)\lambda(\D L) }\\
& \leq \sqrt{\frac{2}{\pi}}\sqrt{1-2 \tau (1-\beta b)||\varphi||^2 _{L^2(\comp{d}, \lambda)}+  \tau^2||\varphi||^4 _{L^2(\comp{d}, \lambda)}}.
\end{align*}
The second term can be estimated analogously:
$$ \int_{\comp{d}} |\varphi(K)|^3 \mathbb{E}[  \lambda^*(K,\mu)] \lambda(\D K) \leq \tau \int_{\comp{d}} |\varphi(K)|^3 \lambda(\D K) \leq \tau ||\varphi||^3 _{L^3(\comp{d}, \lambda)}.  $$
In the following terms, we will use additionally the boundedness of the pair potential $g$. Thus,
\begin{align*}
& \sqrt{\frac{2}{\pi}} \int_{(\comp{d})^2} |\varphi(K)\varphi(L)| \mathbb{E} [| D_K \lambda^*(L, \mu)| \lambda^*(K, \mu)] \lambda(\D K)\lambda(\D L)\\
& = \sqrt{\frac{2}{\pi}} \int_{(\comp{d})^2} |\varphi(K)\varphi(L)| | 1-\e^{-\beta g(K \cap L)}|\alpha_2(K,L,\mu) \lambda(\D K)\lambda(\D L)\\
& \leq \tau^2  |1-\e^{-\beta a}|\sqrt{\frac{2}{\pi}} \int_{(\comp{d})^2} |\varphi(K)\varphi(L)| \lambda(\D K)\lambda(\D L)\\
& = \tau^2 |1-\e^{-\beta a}| \sqrt{\frac{2}{\pi}} ||\varphi||^2 _{L^1(\comp{d}, \lambda)} 
\end{align*}
and
\begin{align*}
& 2 \int_{(\comp{d})^2} |\varphi(K)|^2|\varphi(L)| \mathbb{E} [| D_K \lambda^*(L, \mu)| \lambda^*(K, \mu)] \lambda(\D K)\lambda(\D L)\\
 & = 2 \int_{(\comp{d})^2} |\varphi(K)|^2|\varphi(L)| | 1-\e^{-\beta g(K \cap L)}|\alpha_2(K,L,\mu)  \lambda(\D K)\lambda(\D L)\\
& \leq 2 \tau^2 | 1-\e^{-\beta a}| \int_{(\comp{d})^2} |\varphi(K)|^{2}|\varphi(L)| \lambda(\D K)\lambda(\D L) \\
& = 2 \tau^2 | 1-\e^{-\beta a}| ||\varphi||^2 _{L^{2}(\comp{d}, \lambda)} ||\varphi|| _{L^1(\comp{d}, \lambda)}.
\end{align*}
and
\begin{align*}
&  \int_{(\comp{d})^3} |\varphi(K)\varphi(L)\varphi(M)| \mathbb{E} [| D_K \lambda^*(L, \mu) D_K \lambda^*(M, \mu)| \lambda^*(K, \mu)] \lambda(\D K)\lambda(\D L)\lambda(\D M)\\
&= \int_{(\comp{d})^3} |\varphi(K)\varphi(L)\varphi(M)| | 1-\e^{-\beta g(K \cap L)}|| 1-\e^{-\beta g(K \cap M)}|\alpha_3(K,L,M,\mu) \lambda(\D K)\lambda(\D L)\lambda(\D M)\\
& \leq \tau^3 | 1-\e^{-\beta a}|^2  \int_{(\comp{d})^3} |\varphi(K)\varphi(L)\varphi(M)| \lambda(\D K)\lambda(\D L)\lambda(\D M) \\
&= \tau^3 | 1-\e^{-\beta a}|^2 ||\varphi||^3 _{L^1(\comp{d}, \lambda)}.\\
\end{align*}
Adding these estimates together yields the theorem.
\end{proof}

\subsection{Gaussian approximation for a functional of a stationary Gibbs planar segment process}
As an example of an application of Theorem \ref{th3.4}, we will derive a Gaussian approximation for an innovation of a stationary Gibbs planar segment process defined in Example \ref{ex2.4}. 
Two functionals are investigated:
 the normalized number of segments observed in a window and normalized total length of segments hitting the window. We take windows forming a convex averageing sequence (cf. \cite{DVJ03}), i.e. monotone increasing sequence of convex bounded Borel sets converging to $\R^2$.

\begin{theorem}\label{th3.6}
Consider for each $n \in \N$ a stationary Gibbs planar segment process $\xi^{(n)}$ with the conditional intensity
$$
\lambda^* _n(K, \x) = \tau_n \exp \left\lbrace - \beta_n \sum_{L\in\x} \textbf{1}\lbrace K\cap L \neq \emptyset\rbrace  \right\rbrace,\quad K \in S, \x \in \textbf{N}^2,
$$
where $\tau_n >0$ and $\beta_n\geq 0$. Moreover, suppose that $\beta_n\rightarrow 0$ and $0 < c_1 < \tau_n <c_2 <\infty, \ n \in \N,$ for some constants $c_1,c_2$ and that the common reference particle distribution $\BQ $ for all $\xi^{(n)}$ has a uniform directional distribution. Let $\lbrace W_n, \ n \in \N\rbrace$ be a convex averaging sequence in $\R^2$ such that $Leb(W_n)= O(\beta_n ^{-1})$ (i.e. the asymptotic order of the growth of $Leb(W_n)$ is at most $\beta_n ^{-1}$). For $n \in \N$ and $K \in S$, define 
$$\varphi_n(K)=\frac{1}{\sqrt{\tau_n Leb(W_n)}}\cdot \textbf{1}\lbrace K\cap W_n\neq\emptyset\rbrace.$$ 
Further
$$\psi_n(K)= \frac{l(K)}{\sqrt{\mathbb{E}_L l^2}} \varphi_n(K),$$
where $l(K)$ denotes the length of the segment $K$, $l$ is a random variable that follows the law of $\mathbb{Q}_L$ and $\mathbb{E}_L$ denotes the expectation with respect to $\mathbb{Q}_L$.
Then
$$d_W(I_{\xi^{(n)}}(\varphi_n), Z) \to 0, \qquad d_W(I_{\xi^{(n)}}(\psi_n), Z) \to 0$$
as $n \to \infty $, where $Z$ is a standard Gaussian random variable.
\end{theorem}
\begin{proof}
We prove the theorem for the functions $\psi_n$. The result for $\varphi_n$ is then a special case taking $l(K) = 1,\ K \in S,$ and $\mathbb{E}_L l^2=1$.
We want to use Theorem \ref{th3.4} for $S \subset \comp{2}$. First, we have to verify the assumptions. In Lemma \ref{lemma3.1}, we can set $b = 4\pi R^2$ and $a=1$. Further, for every $n \in \N$,
\begin{align*}
\int_{\comp{2}} |\psi_n(K)|& \lambda(\D K)  = \int_{\comp{2}} l(K) \frac{\textbf{1}\lbrace K \cap W_n \neq \emptyset \rbrace}{\sqrt{\tau_n Leb(W_n)\mathbb{E}_L l^2}} \lambda(\D K)\\
& = \frac{1}{\sqrt{\tau_n Leb(W_n)\mathbb{E}_L l^2}} \int_{S_0}  \int_{\R^2} l(K + x) \textbf{1}\lbrace (K+ x) \cap W_n \neq \emptyset \rbrace \D x \mathbb{Q}(\D K) \\
& = \frac{1}{\sqrt{\tau_n Leb(W_n)\mathbb{E}_L l^2}} \int_{S_0} l(K) \int_{\R^2}  \textbf{1}\lbrace (K + x) \cap W_n \neq \emptyset \rbrace \D x \mathbb{Q}(\D K) \\
& = \frac{1}{\sqrt{\tau_n Leb(W_n)\mathbb{E}_L l^2}} \int_{S_0} l(K) Leb(\check{K} \oplus W_n) \mathbb{Q}(\D K) < \infty,
\end{align*}
since $W_n$ is bounded and $K$ is the segment of the length less than or equal to $2R$. Similarly,
\begin{align*}
\int_{\comp{2}} |\psi_n(K)|^2 \lambda(\D K) & = \int_{\comp{2}}l(K)^2 \frac{\textbf{1}\lbrace K \cap W_n \neq \emptyset \rbrace}{\tau_n Leb(W_n)\mathbb{E}_L l^2} \lambda(\D K)\\
& = \frac{1}{\tau_n Leb(W_n)\mathbb{E}_L l^2} \int_{S_0} l(K)^2 \int_{\R^2} \textbf{1}\lbrace (K + x) \cap W_n \neq \emptyset \rbrace \D x \mathbb{Q}(\D K) \\
& = \frac{1}{\tau_n Leb(W_n)\mathbb{E}_L l^2} \int_{S_0} l(K)^2 Leb(\check{K} \oplus W_n) \mathbb{Q}(\D K) < \infty.
\end{align*}
Hence, the assumptions of Theorem \ref{th3.4} are satisfied and so we can compute the explicit bounds on the Wasserstein distance between a Gaussian random variable $Z$ and the innovation $I_{\xi^{(n)}}(\psi_n)$ for each $n \in \N$.

Take some fixed $n \in \N$ and $\alpha > 1$. Using definition of the measure $\lambda$ and Steiner theorem (cf. \cite{Se82}), we obtain
\begin{align*}
||\psi_n|| _{L^{\alpha}(\comp{2}, \lambda)} & =
\left( \int_{\comp{2}} \abs{l(K) \frac{\textbf{1}\lbrace K \cap W_n \neq \emptyset \rbrace}{\sqrt{\tau_n  Leb(W_n)\mathbb{E}_L l^2}}}^{\alpha} \lambda(\D K)\right) ^{\frac{1}{\alpha}} \\
& = \frac{1}{\sqrt{\tau_n Leb(W_n)\mathbb{E}_L l^2}}\left(  \int_{S_0} l(K)^{\alpha} Leb(\check{K} \oplus W_n) \mathbb{Q}(\D K) \right)^{\frac{1}{\alpha}}\\
& = \frac{1}{\sqrt{\tau_n Leb(W_n)\mathbb{E}_L l^2}}\left( \int_0 ^{2R} \int_{\mathbb{S}^1} r^{\alpha} Leb(\check{K} \oplus W_n) \mathbb{Q}_{\phi}(\D \phi) \mathbb{Q}_L (\D r)\right)^{\frac{1}{\alpha}}\\
& = \frac{1}{\sqrt{\tau_n Leb(W_n)\mathbb{E}_L l^2}}\left( \int_0 ^{2R} r^{\alpha} \left( Leb(W_n) + \frac{r}{\pi} U(W_n)\right) \mathbb{Q}_L(\D r) \right)^{\frac{1}{\alpha}}\\
& = \frac{1}{\sqrt{\tau_n Leb(W_n)\mathbb{E}_L l^2}}\left( Leb(W_n)\mathbb{E}_L l^{\alpha} + \frac{U(W_n)}{\pi} \mathbb{E}_L l^{\alpha + 1} \right)^{\frac{1}{\alpha}},
\end{align*}
where $U(W_n)$ denotes the perimeter of the set $W_n$. Note that since $\mathbb{Q}_L$ has a compact support, it has all moments finite. Then the bound in Theorem \ref{th3.4} can be evaluated as
\begin{align*}
& d_W (I_{\xi^{(n)}}(\psi_n), Z) \leq \sqrt{\frac{2}{\pi}}\sqrt{1-2 \tau_n (1-\beta_n b) ||\psi_n||^2 _{L^2(\comp{2}, \lambda)} + \tau_n ^2 ||\psi_n||^4 _{L^2(\comp{2}, \lambda)}} \\
& + \tau_n ||\psi_n||^3 _{L^3(\comp{2}, \lambda)} + \sqrt{\frac{2}{\pi}} \tau_n ^2 ||\psi_n||^2 _{L^1(\comp{2}, \lambda)} |1-\e ^{-\beta_n}| \\
& +  2 \tau_n ^2 ||\psi_n||^2 _{L^{2}(\comp{2}, \lambda)} ||\psi_n|| _{L^1(\comp{2}, \lambda)} |1-\e ^{-\beta_n}| + \tau_n ^3||\psi_n||^3 _{L^1(\comp{2}, \lambda)} |1-\e ^{-\beta_n}|^2\\
& = \sqrt{\frac{2}{\pi}} \sqrt{1-2 (1-\beta_n b) \left(1 + \dfrac{1}{\pi}\dfrac{U(W_n)}{Leb(W_n)}\dfrac{\mathbb{E}_L l^{3}}{\mathbb{E}_L l^2} \right)+ \left(1 + \dfrac{1}{\pi}\dfrac{U(W_n)}{Leb(W_n)}\dfrac{\mathbb{E}_L l^{3}}{\mathbb{E}_L l^2} \right)^2}\\
& + \dfrac{1}{\sqrt{\tau_n (\mathbb{E}_L l^2)^{3}}}  \left( \dfrac{1}{\sqrt{Leb(W_n)}}\mathbb{E}_L l^{3} + \dfrac{1}{\pi}\dfrac{U(W_n)}{Leb(W_n)^{3/2}} \mathbb{E}_L l^{4} \right)\\
& + \sqrt{\dfrac{2}{\pi}} \dfrac{\tau_n}{\mathbb{E}_L l^2} |1-\e ^{-\beta_n}| \left(\sqrt{Leb(W_n)} \mathbb{E}_L l + \dfrac{1}{\pi}\dfrac{U(W_n)\mathbb{E}_L l^2}{\sqrt{Leb(W_n)}} \right)^2\\
& + 2 \sqrt{\tau_n} |1-\e ^{-\beta_n}| \left( 1 + \dfrac{1}{\pi}\dfrac{U(W_n)}{Leb(W_n)} \dfrac{\mathbb{E}_L l^{3}}{\mathbb{E}_L l^2} \right)\left(\sqrt{Leb(W_n)}\dfrac{\mathbb{E}_L l}{\sqrt{\mathbb{E}_L l^2}}  + \dfrac{1}{\pi}\dfrac{U(W_n)\sqrt{\mathbb{E}_L l^2}}{\sqrt{Leb(W_n)}} \right) \\
& + \tau_n ^{3/2} |1-\e ^{-\beta_n}|^2 \left( \sqrt{Leb(W_n)} \dfrac{\mathbb{E}_L l}{\sqrt{\mathbb{E}_L l^2}} + \dfrac{1}{\pi}\dfrac{U(W_n)}{\sqrt{Leb(W_n)}}\mathbb{E}_L l^{2} \right)^{3}.
\end{align*}
The convexity of $W_n$ implies ${U(W_n)}/{Leb(W_n)}\to 0$ as $n \to \infty $. Combined with the assumed growth of $Leb(W_n)$, also $d_W (I_{\xi^{(n)}}(\psi_n), Z) \to 0$ as $n$ approaches $+\infty$. 
\end{proof}

The assumption of $\beta_n \to 0$ as $n \to \infty$ in Theorem \ref{th3.6} is limiting, analogously to the assumption of $r = 1/n$ in Example $5.9$ in \cite{Tor17}, where $r$ was the hard-core distance. It says that the interactions tend to zero in the sequence of processes investigated. Up to our opinion the presented methodology does not enable to~relax the~assumption $\beta_n \to 0,$ it is an open problem for further research. Also we are able to provide Gaussian approximation for functionals of type $\sum_{x\in\mu}\phi (x)$ here and not for interaction functionals of type $\sum_{x\in\mu}\phi (x,\mu ),$ e.g. the total number of intersections of segments in the window. Generalization of other approaches, e.g. that of \cite{BYY16}, to the space of compact sets, seems to be promising.

\subsection*{Acknowledgement}
\small
This work was supported by Czech Science Foundation, project 16-03708S, and by Charles University, project SVV 2017 No. 260454. 
The authors wish to thank to professor Guenter Last (Karlsruhe Institute of Technology) for suggestions concerning the model and literature, and for stimulating discussions.


\end{document}